\documentclass[11pt,a4paper]{article}
\usepackage{graphicx}
\usepackage{url}
\usepackage{amsmath, amsthm, amssymb}
\usepackage{float}
\usepackage[multiple]{footmisc}
\usepackage{authblk}
\usepackage{fullpage}

\newlength{\bibitemsep}
\newlength{\bibparskip}
\let\oldthebibliography\thebibliography
\renewcommand\thebibliography[1]{%
  \oldthebibliography{#1}%
  \setlength{\parskip}{0pt}%
  \setlength{\itemsep}{0pt plus 0.1ex}%
}

\numberwithin{equation}{section}

\renewcommand\baselinestretch{1.2}

\newtheorem{thm}{Theorem}[section]

\newtheorem{lem}[thm]{Lemma}

\theoremstyle{definition}
\newtheorem{rem}[thm]{Remark}

\theoremstyle{definition}

\newenvironment{eq}
{ \begin{equation} 
  }
{ \end{equation}     } 

\newenvironment{ew}
{ \begin{equation*} 
  }
{ \end{equation*}     }

\begin{document}

\title{Error bounds on the probabilistically optimal \\ problem solving strategy}
\author{Frantisek Duris\thanks{fduris@dcs.fmph.uniba.sk}}
\affil{Department of Computer Science, \\ Faculty of Mathematics, Physics and Informatics, \\ Comenius University, \\ Mlynska dolina 842 48 Bratislava}

\maketitle

\begin{center}
\textbf{Abstract}
\end{center}

We consider a simple optimal probabilistic problem solving strategy that searches through potential solution candidates in a specific order. We are interested in what impact has interchanging the order of two solution candidates with respect to this optimal strategy on the problem solving time. Such interchange (i.e., error) can happen in the applications with only partial information available. We derive bounds on the excessive problem solving time due to these errors in general case as well as in some special cases in which we impose restrictions on the solution candidates. 

~

~

\vfill
\eject

\section{Introduction and layout of the paper}

Problem solving is probably the most common mental activity of humans, and if artificial intelligence is to succeed in creating human like agents, it must deal with the design of general problem solving algorithm. One reason why the human like agents are desirable in the first place is the effectivity with which the human solvers manage to solve vast amount of problems. Moreover, the same human brain (i.e., the same underlying mechanisms) can solve problems from a very wide variety (mathematics, logic, pattern recognition, orientation in space, coordination to name just a few). Solomonoff \cite{sol} in his work on general system for solving problems described a basic model of problem solving process.

There is a casino with a set of bets each winning the same prize. The k$^{th}$ bet has probability of winning $p_k$ and costs $d_k$ dollars. All probabilities are independent, and one cannot make the same bet twice. The probabilities $p_k$ do not need to sum to 1 (i.e., there can be three bets with 50\% chance of winning). If all bets cost equally (e.g. one dollar), then, obviously, the best strategy is to take the bet with highest win probability available. If not all $d_k$ are same, then by selecting the bet with highest ratio $p_k \slash d_k$ available the expected money spent before winning will be minimal (proof given below). By changing dollars $d_k$ to time $t_k$ we get the least expected time before winning.

The problem solving interpretation of this scenario is straightforward, and it outlines the optimal probabilistic problem solving strategy provided we know the values $p_k$ and $t_k$ for all solution candidates (i.e., probability of solving the problem and time of execution for each solution candidate). Admittedly, this requirement of knowing $p_k$ and $t_k$ for all solution candidates is very strong and generally hard to meet. However, in our paper \emph{Theory of the effectivity of human problem solving}  submitted for publication we argued that humans have access to these values or at least to their approximations and that they use them (consciously or unconsciously) to drive their problem solving according to the Solomonoff optimal strategy. It is for this reason we are interested in the Solomonoff optimal problem solving strategy. Moreover, since humans do not always follow the Solomonoff optimal problem solving strategy (inaccurate approximations of $p_k$ or $t_k$, or other reasons), we are interested in what effect the interchange of two candidates has on the problem solving time.

The importance of automatic selection of the problem solving methods the has long been noted \cite{polya1957solve, russell1995provably, hansen1996monitoring}. For example, Fink \cite{fink2004automatic} developed a statistical method for automatic selection among methods based on analysis of their past performance. More particularly, Fink analyzed the expected gain (there is a reward for solving the problem) of each available method based on its past success', failures and the corresponding execution times. The proposed system for solving problems not only chooses the most efficient method (given its limited available data and a time bound), but also decides when to abandon the chosen method if it takes too much time.

Solomonoff' in his work on the general system for solving problems used his theory of inductive inference \cite{sol2, sol4} to approximate the probabilities $p_k$ and Levin search \cite{levin1973universal, solomonoff1984optimum} to approximately follow the optimal order of examination of candidates as outlined in his betting scenario \cite{sol}. Since this is an approximation of the optimal order, it does not need to follow it exactly, and we might be interested in what effect the shuffling of candidates has on the problem solving time. Other problem solving systems that are also related to the Solomonoff problem solving strategy and the approximation of probabilities $p_k$ (e.g., AIXI proposed by Hutter \cite{hutter2000theory, marcus2005universal}) may benefit from this paper in the similar way. 

Schmidhuber \cite{oops} described a similar notion of time optimality of the problem solving. While he also uses the values $p_k$ and $t_k$, he does not try to approximate the probabilities $p_k$ (termed initial bias) as Solomonoff. Instead, he assumes that they are given as input (although they may change during the problem solving process; the same applies to the Solomonoff problem solving system). Again, the changes in the initial probabilities may lead to different order of candidate examination, and it is interesting to ask what effect does this have on the problem solving time.

The layout of the paper is as follows. In \emph{Section \ref{sec:two}} we prove the optimality of the strategy mentioned in the introduction as Solomonoff \cite{sol} did not do so. The proof will serve as a basis for all our subsequent results. In \emph{Section \ref{sec:three}} we consider the effect of interchanging two candidates with respect to the optimal strategy on the problem solving time. In \emph{Section \ref{sec:four}} we give several bounds on the error resulting from the mentioned interchange. However, since the values $p_k$ and $t_k$ can be arbitrary, we examine special cases under which reasonable bounds can be achieved. In \emph{Section \ref{sec:five}} we consider a modification of this strategy when each solution candidate can have multiple values of execution time (i.e., multiple values $t_k$). This modification models the case when we applied the same solution candidate (e.g., a method) to two or more similar problems each time solving the problem in different time.

\section{Theorem in probability} \label{sec:two}

Let there be a set of bets each winning the same prize. Let $p_k$ denote a probability of winning with $k^{th}$ bet, and let $d_k$ denote the cost of this bet. Each bet can only be taken once, and the values $p_k$ do not need to sum to 1. One strategy to win is to take bets in the order given by decreasing value $\frac{p_k}{d_k}$. This strategy appeared in Solomonoff paper on general problem solving system \cite{sol}, but the proof was omitted. In this section we give proof of this statement because it will serve as a basis for all of our subsequent results. Additionally, since we would like to apply this strategy in real world, we restrict the number of bets to some finite number $N$.

\begin{thm}[Solomonoff \cite{sol}]
Let there be a set of $N$ bets $\{s_k\}_{k=1}^N$, each with probability of winning $p_k$ and cost $d_k$. If one continues to select subsequent bets on the basis of maximum $p_k/d_k$, the expected money spent before winning will be minimal. Suppose we change dollars to some measure of time $t_k$. Then, betting according to this strategy yields the minimum expected time to win.
\label{inf:sol2}
\end{thm}

\begin{rem}
Note that if the bets are selected in the order: 1$^{st}$, 2$^{nd}$,..., $N^{th}$, then the probability of using and winning with a particular bet $k$ is not $p_k$ but $\prod_{i=1}^{k-1}(1-p_i)\cdot p_k$. This is because in order to make and win with the $k^{th}$ bet all bets with the indices $1,2,...,k-1$ must have failed.
\end{rem}

\begin{proof}[Proof of the Theorem \ref{inf:sol2}]
Without loss of generality, we may assume that the sequence of bets $(s_k)_{k=1}^N$ is ordered according to the values $\frac{p_k}{t_k}$ in the decreasing order, i.e.,
\begin{eq}
\frac{p_1}{t_1}\geq \frac{p_2}{t_2}\geq \dots \geq \frac{p_k}{t_k} \geq\dots\geq \frac{p_N}{t_N},
\label{eq:order}
\end{eq}
in which case the Solomonoff strategy $SOL$ is $(s_k)_{k=1}^{N}$. Let $E_T$ be the expected time spent before winning using the Solomonoff strategy $SOL$. Clearly,
\begin{ew}
E_T=\sum_{k=1}^{N}{\sum_{l=1}^{k}{t_l}\cdot\prod_{j=1}^{k-1}{(1-p_j)}\cdot p_k}.
\end{ew}
We want to show that this strategy is optimal (with respect to the time spent before winning). Let $ABC=(s_{i_k})_{k=1}^{N}$ be any betting strategy (i.e., a sequence of bets). Furthermore, let $E_{ABC}$ be the expected time spent before winning for the strategy $ABC$. Clearly, 
\begin{ew}
E_{ABC}=\sum_{k=1}^{N}{\sum_{l=1}^{k}{t_{i_l}}\cdot\prod_{j=1}^{k-1}{(1-p_{i_j})}\cdot p_{i_k}}.
\end{ew}
Our aim is to show that $E_{ABC}\geq E_T$. If $ABC=SOL$, then we have nothing to prove. Now assume that there are two immediately subsequent bets $s_{i_a}$ and $s_{i_{a+1}}$ in the sequence $ABC$ such that $\frac{p_{i_a}}{t_{i_a}}< \frac{p_{i_{a+1}}}{t_{i_{a+1}}}$. The case when $\frac{p_{i_a}}{t_{i_a}} \geq \frac{p_{i_{a+1}}}{t_{i_{a+1}}}$ for each $a\in\{1,2,...,N-1\}$ but $ABC\neq SOL$ will be considered below. Let $ABC'$ be a modified sequence $ABC$ in which the terms $s_{i_{a}}$ and $s_{i_{a+1}}$ are interchanged. We will show that $E_{ABC}\geq E_{ABC'}$ where $E_{ABC'}$ denotes the analogous value for $ABC'$. First of all, notice that all terms in $E_{ABC}$ and $E_{ABC'}$ are equal except for the terms on the $a^{th}$ and $(a+1)^{th}$ position. In the expression $E_{ABC}$ we have the following value related to these two positions
\begin{ew}
\sum_{l=1}^{a}t_{i_l}\cdot\prod_{j=1}^{a-1}(1-p_{i_j})\cdot p_{i_a} + \sum_{l=1}^{a+1}t_{i_l}\cdot\prod_{j=1}^{a}(1-p_{i_j})\cdot p_{i_{a+1}},
\end{ew}
while in the expression $E_{ABC'}$ there is
\begin{ew}
\left(\sum_{l=1}^{a-1}t_{i_l} + t_{i_{a+1}}\right)
\cdot\prod_{j=1}^{a-1}(1-p_{i_j})\cdot p_{i_{a+1}} + 
\left(\sum_{l=1}^{a-1}t_{i_l} + t_{i_{a+1}} + t_{i_a}\right)
\cdot\prod_{j=1}^{a-1}(1-p_{i_j})\cdot (1-p_{i_{a+1}}) p_{i_{a}}.
\end{ew}
Therefore,
\begin{align*}
E_{ABC}-E_{ABC'} &=
\left(\sum_{l=1}^{a-1}{t_{i_l}}+t_{i_a}\right)\cdot
\prod_{j=1}^{a-1}{(1-p_{i_j})}\cdot
p_{i_a} -\\
&-\left(\sum_{l=1}^{a-1}{t_{i_l}}+t_{i_{a+1}}\right)\cdot
\prod_{j=1}^{a-1}{(1-p_{i_j})}\cdot p_{i_{a+1}} +\\
&+\left(\sum_{l=1}^{a-1}{t_{i_l}}+t_{i_a}+t_{i_{a+1}}\right)\cdot
\prod_{j=1}^{a-1}(1-p_{i_j})(1-p_{i_a})p_{i_{a+1}} -\\
&-\left(\sum_{l=1}^{a-1}{t_{i_l}}+t_{i_{a+1}} + t_{i_a}\right)
\cdot\prod_{j=1}^{a-1}(1-p_{i_j})(1-p_{i_{a+1}})p_{i_{a}}.
\end{align*}
By substituting
\begin{ew}
\sum =  \sum_{l=1}^{a-1}{t_{i_l}} ~~~\text{and}~~~ \prod = \prod_{j=1}^{a-1}(1-p_{i_j})
\end{ew}
and some rearranging we get
\begin{align*}
E_{ABC}-E_{ABC'}&= \sum\cdot\prod\cdot\left(p_{i_a} - p_{i_{a+1}}\right) +\\
&+ \prod\cdot\left(t_{i_a}p_{i_a} - t_{i_{a+1}}p_{i_{a+1}}\right)+ \\
&+\sum\cdot\prod\cdot\left(p_{i_{a+1}} - p_{i_a}\right) +\\
&+ \prod\cdot\left(t_{i_a}+t_{i_{a+1}}\right)\left(p_{i_{a+1}}-p_{i_a}\right)=\\
&=\prod\cdot\left(t_{i_a}p_{i_{a+1}} - t_{i_{a+1}}p_{i_a}\right)=\\
&=\prod_{j=1}^{a-1}(1-p_{i_j})\cdot t_{i_a}t_{i_{a+1}}\left(\frac{p_{i_{a+1}}}{t_{i_{a+1}}} - \frac{p_{i_a}}{t_{i_a}}\right)\geq 0.
\end{align*}
It follows that $ABC$ can be turned into a betting strategy $DEF=(s_{m_k})_{k=1}^N$ with the property $\frac{p_{m_1}}{t_{m_1}}\geq\frac{p_{m_2}}{t_{m_2}}\geq...\geq \frac{p_{m_N}}{t_{m_N}}$ by repeatedly modifying the obtained sequences. Moreover, the expected amount of dollars paid over a sequence (i.e., strategy) before winning is not increased after its modification. Thus, $E_{ABC}\geq E_{ABC'}\geq ...\geq E_{DEF}$. If $DEF \neq SOL$, we proceed as follows. Let $DEF'$ be a modified sequence $DEF$ in which any two terms $s_{m_b}$ and $s_{m_{b+1}}$ with $\frac{p_{m_b}}{t_{m_b}} = \frac{p_{m_{b+1}}}{t_{m_{b+1}}}$ are interchanged (call this kind of modification by simple modification). Using the same calculation as for $E_{ABC}-E_{ABC'}$ above, we get $E_{DEF}-E_{DEF'}=0$, where $E_{DEF'}$ denotes analogous value for $DEF'$, since $\frac{p_{m_b}}{t_{m_b}} - \frac{p_{m_{b+1}}}{t_{m_{b+1}}}=0$. Now one can observe that $DEF$ can be turned into $SOL$ by repeatedly modifying the obtained sequences (using only the simple modification). Thus, $E_{DEF}=E_{DEF'}=...=E_{SOL}$.

Finally, let us consider the case when $\frac{p_{i_a}}{t_{i_a}} = \frac{p_{i_{a+1}}}{t_{i{a+1}}}$ for each $a\in\{1,2,...,N-1\}$, but $ABC \neq SOL$. In such case, we can turn $ABC$ into $SOL$ by the same way as we have turned $DEF$ into $SOL$ above, and therefore $E_{ABC}=E_T$. Consequently, $E_{ABC}\geq E_{DEF}=E_T$ (see above), or $E_{ABC}=E_T$. Hence, $SOL$ is optimal, since the strategy $ABC$ has been chosen arbitrarily.
\end{proof}

\begin{rem}
Notice that the expected betting time is not given by $E_T$ because we did not include the possibility that all of our bets failed. The corrected value $E_T$ is given by
\begin{eq}
E_T=\sum_{k=1}^{N}{\sum_{l=1}^{k}{t_l}\cdot\prod_{j=1}^{k-1}{(1-p_j)}\cdot p_k} + \sum_{l=1}^{N}{t_l}\cdot\prod_{j=1}^{N}{(1-p_j)}.
\label{eq:remark}
\end{eq}
This is because the probability of each bet failing is $\prod_{j=1}^{N}{(1-p_j)}$ while it takes us altogether $\sum_{l=1}^{N}{t_l}$ amount time to discover this.
\label{inf:sol2a}
\end{rem}

\section{General effects of exchanging the order of two candidates on the problem solving effectivity} \label{sec:three}

In this section we abandon the Solomonoff' original betting scenario, and we will talk about solution candidates (i.e., bets) and problem solving strategy (i.e., betting on candidates). We consider and quantify the expected {decrease in problem solving effectivity} when the solver makes an error and {interchanges the order} of two (not necessarily immediately subsequent) candidates with respect to the optimal Solomonoff problem solving strategy.

\begin{lem}[Klamkin and Newman \cite{klam}]
If $x_1, x_2,...,x_n$ are numbers in $[0,1]$ whose sum is denoted by $S$, then
\begin{ew}
\prod_{i=1}^n (1-x_i)<e^{-S}.
\end{ew}
\label{lem_w}
\end{lem}

\begin{lem}[Wu \cite{shan}]
Let $0\leq x_i \leq 1, i=1,2,...,n,n\geq 2, n\in\mathbb{N}$. Then we have
\begin{ew}
\prod_{i=1}^n (1-x_i)\geq  1 - \sum_{i=1}^nx_i + (n-1)\left(\prod_{i=1}^nx_i\right)^{\frac{n}{2n-2}}.
\end{ew}
\label{lem_w2}
\end{lem}

\begin{rem} 
For the rest of this paper we will use the following notation 
%\[ \begin{array}{llll}
%S_m = \sum_{i=1}^m{p_i},  &  T_m = \sum_{i=1}^m{t_i}, & P_m = \prod_{i=1}^m{p_i}, &  Q_m = \prod_{i=1}^m{(1-p_i)}.
%\end{array}\] 
\[ \begin{array}{ll}
S_m = \sum_{i=1}^m{p_i},  &  T_m = \sum_{i=1}^m{t_i}, \\[3mm]
P_m = \prod_{i=1}^m{p_i}, &  Q_m = \prod_{i=1}^m{(1-p_i)}.
\end{array}\] 
\label{rem:not}
\end{rem}

\begin{thm}
Let $\frac{p_k}{t_k} - \frac{p_{k+1}}{t_{k+1}}> 0$ for some $k$ (assuming $(\frac{p_i}{t_i})_{i=1}^N$ to be ordered as before in Theorem \ref{inf:sol2}). Then, following the optimal Solomonoff strategy from Theorem \ref{inf:sol2} with $(k+1)^{th}$ solution candidate tried just before $k^{th}$ (a solver's error) yields a sub-optimal expected amount of time spent before either finding a solution or discovering that none of our solution candidates works, and the expected excess $EXC$ can be quantified as follows
\begin{ew}
EXC = \left(\frac{p_k}{t_k} - \frac{p_{k+1}}{t_{k+1}}\right) \prod_{j=1}^{k-1}(1-p_j) \cdot t_kt_{k+1}.
\end{ew}
Furthermore,
\begin{ew}
\left(\frac{p_k}{t_k} - \frac{p_{k+1}}{t_{k+1}}\right) t_kt_{k+1}\cdot e^{-S_{k-1}}\geq
EXC \geq
\left(\frac{p_k}{t_k} - \frac{p_{k+1}}{t_{k+1}}\right) t_kt_{k+1}\cdot\left(1-S_{k-1} + (k-2)P_{k-1}^{\frac{k-1}{2k-4}}\right).
\end{ew}
\label{inf:err2}
\end{thm}

\begin{proof}
The theorem follows directly from proof of the \emph{Theorem \ref{inf:sol2}} (see the expression \mbox{$E_{ABC}-E_{ABC'}$} and below). To derive the bounds, use \emph{Lemma \ref{lem_w}} and \emph{Lemma \ref{lem_w2}}. For notation see \emph{Remark \ref{rem:not}}.
\end{proof}

\begin{thm}
If $(\frac{p_i}{t_i})_{i=1}^N$ are ordered as before in Theorem \ref{inf:sol2}, then exchanging the $k^{th}$ and $(k+n)^{th}$ solution candidates in the optimal Solomonoff strategy from Theorem \ref{inf:sol2} (a solver's error) increases the expected amount of problem solving time by at most the excess 
\begin{ew}
EXC = q_1 + q_2 + q_3,
\end{ew} 
where
\begin{align*}
q_1 &=  T_{k-1}\cdot Q_{k-1}\cdot (p_{k+n} - p_k) + Q_{k-1}\cdot (t_{k+n}p_{k+n} - t_kp_k), \\[2mm]
q_2 &= \sum_{l=k+1}^{k+n-1} Q_{l-1}\cdot p_l\left(T_l\cdot\frac{p_k-p_{k+n}}{1-p_k} + (t_{k+n} - t_k)\frac{1-p_{k+n}}{1-p_k}\right), \\[2mm]
q_3 &= T_{k+n}\cdot Q_{k+n-1} \cdot\frac{p_k - p_{k+n}}{1-p_k}.
\end{align*}
\label{thm_pr2}
\end{thm}

\begin{proof}
Use the equation (\ref{eq:remark}) and subtract it from an analogous equation in which the $k^{th}$ and $(k+n)^{th}$ candidates are exchanged. The theorem follows by straightforward computation.
\end{proof}

\begin{rem}
The practicality of the \emph{Theorem \ref{inf:err2}} is considerably limited by the fact that the values $(p_i)_{i=1}^n$ and $(t_i)_{i=1}^N$ can be (almost) arbitrary numbers (obviously, we want $p_i\in [0,1]$ and $t_i>0$). Thus, we could not further simplify the terms $q_1, q_2, q_3$. In \emph{Section \ref{sec:four}} we make some assumptions about these numbers to obtain reasonable bounds on the term $ECX$ from \emph{Theorem \ref{inf:err2}}. Note that these assumptions directly affect the quality of the obtained bounds and the form of the resulting formulas. On the other hand, they also reduce the generality of the derived bounds.
\end{rem}

\section{More precise quantification of the effects in some special cases} \label{sec:four}

\begin{lem}
Let $r$ be any real number other than 1. Then it holds
\begin{ew}
\sum_{l=1}^n lr^l = \frac{r}{(1-r)^2}\left(nr^{n+1} - (n+1)r^n +1\right).
\end{ew}
\label{lem1}
\end{lem}

A very general bound of the $EXC$ term from the \emph{Theorem \ref{thm_pr2}} follows. Note that we were not able to derive similarly simple lower bound with the same weak assumptions (see \mbox{\emph{Theorem \ref{thm:new2}}}).

\begin{thm}
For each $i=1,2,...,N$, let $c \leq p_i\leq d$ for some $c,d\in(0,1)$ and $0 \leq t_i \leq T$ for some $T$. Then, the term EXC from the Theorem \ref{thm_pr2} can be upper bounded as follows
\begin{ew}
EXC \leq T \frac{1-p_{k+n}}{1-p_k}\frac{d}{c}\left(1-c\right)^{k}\left(A - B(1-c)^{n-1}\right),
\end{ew}
where
\begin{align*}
A &= 1 + \frac{1}{c} + k + \frac{1-p_k}{1-p_{k+n}} \frac{ck}{1-c}, \\[3mm]
B &= \left(1 - \frac{c}{d}\right)(k+n)  + \frac{1}{c}.
\end{align*}
\label{thm:new1}
\end{thm}

\begin{proof}
Let $q_1$, $q_2$, and $q_3$ be expressions from \emph{Theorem \ref{thm_pr2}}. From $(1-p_i)\leq(1-c)$ for all $i$ we have for $q_1$
\begin{align*}
q_1 &= Q_{k-1}p_{k+n}\left(T_{k-1} + t_{k+n}\right) - Q_{k-1}p_k\left(T_{k-1} + t_{k}\right) \leq \\[3mm]
&\leq Q_{k-1}p_{k+n}\left(T_{k-1} + t_{k+n}\right) \leq \\[3mm]
&\leq (1-c)^{k-1}d k T.
\end{align*}
Finally, we transform $q_1$ into
\begin{eq}
q_1 \leq T \frac{1-p_{k+n}}{1-p_k}\frac{d}{c}\left(1-c\right)^{k} \cdot \frac{1-p_k}{1-p_{k+n}}\frac{ck}{1-c}.
\label{eq:new1}
\end{eq}
For $q_2$ we have
\begin{align*}
q_2 &\leq \sum_{l=k+1}^{k+n-1}Q_{l-1}p_l\left(l T \frac{1-p_{k+n}}{1-p_k} + T\frac{1-p_{k+n}}{1-p_k}\right)=\\[3mm]
&= T\frac{1-p_{k+n}}{1-p_k}\sum_{l=k+1}^{k+n-1}Q_{l-1}(l+1)p_l\leq\\[3mm]
&\leq T\frac{1-p_{k+n}}{1-p_k}d\sum_{l=k+1}^{k+n-1}(1-c)^{l-1}(l+1)=\\[3mm]
&= T\frac{1-p_{k+n}}{1-p_k}d(1-c)^{k-1}\left(\sum_{l=1}^{n-1}(1-c)^l l + (k+1)\sum_{l=1}^{n-1}(1-c)^l\right).
\end{align*}
By using \emph{Lemma \ref{lem1}} we can further write
\begin{align}
\sum_{l=1}^{n-1}(1-c)^l l &= \frac{1-c}{c^2}\left(1 + (c-nc-1)(1-c)^{n-1}\right), \label{eq:lem1}\\[3mm]
(k+1)\sum_{l=1}^{n-1}(1-c)^l &= (k+1)(1-c)\frac{1-(1-c)^{n-1}}{c}. \label{eq:lem2}
\end{align}
Finally, we have for $q_2$
\begin{align}
q_2 &\leq T \frac{1-p_{k+n}}{1-p_k}\frac{d}{c} (1-c)^k
\left(1 + \frac{1}{c} + k + (1-c)^{n-1}\left(\frac{c-nc-1}{c} - k -1\right)\right)=\nonumber\\
&=T \frac{1-p_{k+n}}{1-p_k}\frac{d}{c} (1-c)^k
\left(1 + \frac{1}{c} + k - (1-c)^{n-1}\left(k + n + \frac{1}{c}\right)\right).
\label{eq:new2}
\end{align}
For $q_3$ we have
\begin{ew}
q_3 \leq (k+n)T (1-c)^k (1-c)^{n-1} \frac{1-p_{k+n}}{1-p_k}
\end{ew}
which we transform into
\begin{eq}
q_3 \leq T \frac{1-p_{k+n}}{1-p_k}\frac{d}{c}(1-c)^k \cdot (1-c)^{n-1}\frac{c}{d}(k+n).
\label{eq:new3}
\end{eq}
By putting (\ref{eq:new1}), (\ref{eq:new2}), and (\ref{eq:new3}) together the theorem follows.

\end{proof}

\begin{thm}
For each $i=1,2,...,N$, let $c \leq p_i\leq d$ for some $c,d\in(0,1)$ and $0 \leq t \leq t_i \leq T$ for some $t,T$. Furthermore, let $p_k\geq p_{k+n}$ and $t_{k}\leq t_{k+n}$, where $p_k,p_{k+n},t_k,t_{k+n}$ are from Theorem \ref{thm_pr2}. Then, the term EXC from the Theorem \ref{thm_pr2} can be lower bounded as follows
\begin{ew}
EXC \geq t \frac{p_k-p_{k+n}}{1-p_k}\frac{c}{d}\left(1-d\right)^{k}\left(A - B(1-d)^{n-1}\right),
\end{ew}
where
\begin{align*}
%A &= \frac{1}{d} + k + \frac{1-p_k}{p_k - p_{k+n}}\frac{dk}{1-d} - \frac{1-p_k}{p_k - p_{k+n}}\frac{d^2kt}{cT}\frac{e^{-S_{k-1}}}{(1-d)^k}, \\[3mm]
A &= \frac{1}{d} + k + \frac{1-p_k}{p_k - p_{k+n}}dk\left(\frac{1}{1-d} - \frac{dt}{cT}\frac{e^{-S_{k-1}}}{(1-d)^k}\right), \\[3mm]
B &= \left(1 - \frac{d}{c}\right)(k+n) + \frac{1-d}{d}.
\end{align*}
\label{thm:new2}
\end{thm}

\begin{proof}
Let $q_1$, $q_2$, and $q_3$ be expressions from \emph{Theorem \ref{thm_pr2}}. We rewrite $q_1$ as
\begin{ew}
q_1 = Q_{k-1}\left(T_{k-1} + t_{k+n}\right) p_{k+n} - Q_{k-1}\left(T_{k-1} + t_k\right)p_k.
\end{ew}
By \emph{Lemma \ref{lem_w}} we have $\prod(1-p_i)\leq e^{-\sum p_i}$ and with $\prod(1-p_i)\geq\prod(1-d)$ we can write
\begin{ew}
q_1 \geq (1-d)^{k+1} k t c - e^{-S_{k-1}} k T d
\end{ew}
which we transform into
\begin{ew}
t \frac{p_k-p_{k+n}}{1-p_k}\frac{c}{d}\left(1-d\right)^{k} \left(
\frac{1-p_k}{p_k - p_{k+n}}\frac{dk}{1-d} -
\frac{1-p_k}{p_k - p_{k+n}}\frac{d^2kt}{cT}\frac{e^{-S_{k-1}}}{(1-d)^k}.
\right)
\end{ew}
The rest of the proof is analogous to the proof of the \emph{Theorem \ref{thm:new1}}.
\end{proof}

A different bound of the $EXC$ term from the \emph{Theorem \ref{thm_pr2}} follows. In this case we chose a strong assumption that all $t_i$ (see \emph{Theorem \ref{inf:sol2}}) are equal. This case models a situation where the solver has candidates that take approximately the same amount of time (e.g., a weak/novice solver that only has very general methods to apply in problem solving that equally take a lot of time, or an expert solver that has a set a specific methods with comparable execution time). Note that in this case we state both upper and lower bound with the same assumptions (\mbox{\emph{Theorem \ref{inf:bound32}}} and \emph{Theorem \ref{inf:bound42}}, respectively).

\begin{thm}
For each $i=1,2,...,N$, let $c\leq p_i\leq d$ for some $c,d\in(0,1)$ and $t_i = T$ for some $T>0$. Then, the term EXC in Theorem \ref{thm_pr2} can be upper bounded as
\begin{ew}
EXC \leq T d \frac{p_k - p_{k+n}}{1-p_k} \frac{(1-c)^k}{c^2} \left(A - B(1-c)^{n-1}\right)
\end{ew}
where
\begin{align*}
A &= 1 + kc\left[1 - \frac{1-p_k}{1-c} \frac{c}{d}\left(\frac{1-d}{1-c}\right)^{k-1}\right],\\[3mm]
B &= 1 - c + c(n+k)\left(1 - \frac{c}{d}\right).
\end{align*}
\label{inf:bound32}
\end{thm}

\begin{proof}
First of all, observe that with $t_i=T$ for all $i=1,2,...,N$ the order of candidates from \emph{Theorem \ref{inf:sol2}}, which is given by $(\frac{p_i}{t_i})_{i=1}^N$, depends only on $p_i$. Thus, by (\ref{eq:order}) we have $p_1\geq p_2 \geq ... \geq p_N$, and particularly $p_k\geq p_{k+n}$.

Let $q_1$, $q_2$, and $q_3$ be expressions from \emph{Theorem \ref{thm_pr2}}. With $p_{k+n} - p_k \leq 0$ and $t_i=T$ for each $i=1,2,...,N$ we have for $q_1$
\begin{ew}
q_1 = kT(p_{k+n} - p_k)Q_{k-1},
\end{ew}
and because $-Q_m\leq-(1-d)^m$, we can write
\begin{ew} 
q_1\leq -kT(p_k - p_{k+n}) (1-d)^{k-1}
\end{ew}
which we transform into
\begin{eq}
q_1 \leq -T d \frac{p_k - p_{k+n}}{1-p_k} \frac{(1-c)^k}{c^2} \cdot
k \frac{1-p_k}{d} \frac{c^2}{1-c}\left(\frac{1-d}{1-c}\right)^{k-1}.
\label{eq:new4}
\end{eq}
For $q_2$ we have
\begin{align*}
q_2 &=  T \frac{p_k - p_{k+n}}{1-p_k} \sum_{l=k+1}^{k+n-1}Q_{l-1} l p_l \leq\\[3mm]
&\leq T \frac{p_k - p_{k+n}}{1-p_k} d \sum_{l=k+1}^{k+n-1} (1-c)^{l-1} l = \\[3mm]
&=T d \frac{p_k - p_{k+n}}{1-p_k} (1-c)^{k-1} \left( \sum_{l=1}^{n-1} l (1-c)^{l} + k\sum_{l=1}^{n-1} (1-c)^{l} \right).
\end{align*}
By \emph{Lemma \ref{lem1}} we can use similar equations to (\ref{eq:lem1}) and (\ref{eq:lem2}), and we can further write
\begin{eq}
q_2 \leq T d \frac{p_k - p_{k+n}}{1-p_k} \frac{(1-c)^{k}}{c^2} \cdot
\left(1 + kc - (1-c)^{n-1}\left(1+c(n+k-1)\right)\right).
\label{eq:new5}
\end{eq}
For $q_3$ we have
\begin{align*}
q_3 &= (k+n) T  \frac{p_k - p_{k+n}}{1-p_k} Q_{k+n-1} \leq \\[3mm]
&\leq (k+n)T\frac{p_k - p_{k+n}}{1-p_k}(1-c)^{k+n-1}
\end{align*}
which we transform into
\begin{eq}
q_3 \leq
T d \frac{p_k - p_{k+n}}{1-p_k} \frac{(1-c)^k}{c^2}\cdot \frac{c^2}{d}(k+n)(1-c)^{n-1}.
\label{eq:new6}
\end{eq}
Finally, by putting (\ref{eq:new4}), (\ref{eq:new5}), and (\ref{eq:new6}) together the theorem follows.
\end{proof}

\begin{thm}
For each $i=1,2,...,N$, let $c\leq p_i \leq d$ for some $c,d\in(0,1)$ and $t_i = T$ for some $T>0$. Then, the term EXC from the Theorem \ref{thm_pr2} can be lower bounded as
\begin{ew}
EXC \geq T c \frac{p_k - p_{k+n}}{1-p_k} \frac{(1-d)^k}{d^2} \left(A - B(1-d)^{n-1}\right)
\end{ew}
where
\begin{align*}
A &= 1 + kd\left[1 - \frac{1-p_k}{1-d}\frac{d}{c}\frac{e^{-S_{k-1}}}{\left(1-d\right)^{k-1}}\right],\\[3mm]
B &= 1 - d + d(n+k)\left(1-\frac{d}{c}\right).
\end{align*}
\label{inf:bound42}
\end{thm}

\begin{proof}
By the same reason as in the beginning of the proof of the \emph{Theorem \ref{inf:bound32}} we have \mbox{$p_k\geq p_{k+n}$}. Let $q_1$, $q_2$, and $q_3$ be expressions from \emph{Theorem \ref{thm_pr2}}. With $p_{k+n} - p_k \leq 0$ and $t_i=T$ for each $i=1,2,...,N$ we have for $q_1$
\begin{ew}
q_1 = kT(p_{k+n} - p_k)Q_{k-1},
\end{ew}
and by \emph{Lemma \ref{lem_w}} we can write
\begin{ew} 
q_1 \geq - kT(p_k - p_{k+n})e^{-S_{k-1}}
\end{ew}
which we transform into
\begin{ew}
q_1 \geq -T c \frac{p_k - p_{k+n}}{1-p_k} \frac{(1-d)^k}{d^2} \cdot
k \frac{1-p_k}{1-d} \frac{d^2}{c}\frac{e^{-S_{k-1}}}{\left(1-d\right)^{k-1}}.
\end{ew}
The rest of the proof is analogous to the proof of the \emph{Theorem \ref{inf:bound32}}.
\end{proof}

Finally, we examine a case in which the probabilities $p_i$ (see \emph{Theorem \ref{inf:sol2}}) are all equal. This case models, for example, a situation where the solver has only general problem solving methods, and he is unable to distinguish which one is the most suitable.
\begin{thm}
For each $i=1,2,...,N$, let $p_i = p$  for some $p\in(0,1)$. Then, the term EXC from the Theorem \ref{thm_pr2} equals to
\begin{ew}
EXC = (t_{k+n} - t_k) (1-p)^{k-1} \left(1 + (1-p) - (1-p)^{n}\right).
\end{ew}
\label{inf:ind}
\end{thm}

\begin{proof}
Let $q_1$, $q_2$, and $q_3$ be expressions from \emph{Theorem \ref{thm_pr2}}. Then, 
\begin{align*}
q_1 &= Q_{k-1}(t_{k+n} - t_k) =  (t_{k+n} - t_k)(1-p)^{k-1},\\[3mm]
q_2  &= \sum_{l=k+1}^{k+n-1}Q_{l-1}p(t_{k+n} - t_k) = \\
&= p (t_{k+n} - t_k) (1-p)^{k-1} \sum_{l=1}^{n-1} (1-p)^l =\\
&= p (t_{k+n} - t_k) (1-p)^{k}\frac{1 - (1-p)^{n-1}}{p}, \\[3mm]
q_3 &= 0.
\end{align*}
By summing the expressions for $q_1,q_2,q_3$ we get the final result.
\end{proof}

\section{The case with multiple values of $t_k$} \label{sec:five}

In real life problem solving a particular solution candidate (e.g., a method) could have been used to solve multiple similar problems each time consuming {a different amount of time}. Therefore, when the solver is considering a potential solution candidate, it has one cumulative probability of success (e.g., based on the past experience and the strength of similarity/relatedness with the current problem model), but it can have multiple application times because of this possible application to the similar problems in the past. 

\begin{thm}
Let $s_k$ be a solution candidate which we in the past applied $n_k$ times, and let $t_{k,j}$ be the execution time of the $j^{th}$ application. Denote the mean execution time of the solution candidate $s_k$ with $Et_k$:
\begin{ew}
Et_k = \frac{t_{k,1} + ... + t_{k,n_k}}{n_k}.
\end{ew}
If one continues to select subsequent candidates on the basis of maximum $p_k/Et_k$, then the expected time before solving the problem will be minimal (provided the problem can be solved by one of our candidates).
\label{inf:sol3}
\end{thm}

\begin{proof}
By using the linearity of the expected value, the proof is identical to the proof of the \emph{Theorem \ref{inf:sol2}}.
\end{proof}

\addcontentsline{toc}{section}{References}
\renewcommand\baselinestretch{1}
\bibliographystyle{abbrv}
\bibliography{bibi}

\end{document}